\documentclass[12pt]{article}
\usepackage{ayv_paperstyle}

\title{Krasovskii-Subbotin approach to mean field type differential games}
\author{Yurii Averboukh\thanks{Krasovskii Institute of Mathematics and Mechanics,\ \ \texttt{e-mail: ayv@imm.uran.ru, averboukh@gmail.com}}{ }\thanks{
		Ural Federal University}}

\date{}
\begin{document}
\maketitle
	
\begin{abstract}
	A  mean field type differential game is a mathematical model of a large system of  identical agents under mean-field interaction controlled by two players with opposite purposes. We study the case when the dynamics of each agent is given by ODE and the players can observe the distribution of the agents. We construct suboptimal strategies and prove the existence of the value function. 
	\msccode{49N70, 91A24, 49N35, 60K35}
\end{abstract}	

\section{Introduction}
The paper is concerned with the deterministic mean field type control system governed by two players. We assume that the purposes of the players are opposite. This problem can be called zero-sum mean field type differential game.

Originally, the theory of differential games deals with the finite dimensional  systems controlled by two players. It started with the seminal work by Isaacs~\cite{isaacs}. The mathematical analysis of zero-sum differential games was developed in 1970s (see \cite{bardi}, \cite{elliot_kalton}, \cite{friedman}, \cite{NN_PDG_en}, \cite{varaya_lin}  and references therein). There are several equivalent approaches to the formalization of the zero-sum differential games. First one developed by Krasovskii, Subbotin and their followers (see~\cite{NN_PDG_en},~\cite{Subb_book}) presumes that the players are informed about the current state of the system and form their controls stepwise. Krasovskii and Subbotin  called this approach feedback. However, in fact the realization of such strategy requires a short-term memory. Within the framework of the second approaches it is assumed that the player's strategy is a nonanticipative response on a control of his/her partner \cite{elliot_kalton}, \cite{varaya_lin}. 
Note that  value function of the zero-sum differential game is a viscosity (minimax) solution to the corresponding Hamilton-Jacobi equation \cite{bardi}, \cite{Subb_book}. 

The  mean field type control  theory is concerned with  large systems of agents optimizing their common payoff in the limit case when the number of agents tends to infinity.  The main assumptions of the mean field approach are: the agents are identical and the dynamics of each agent depends on his/her state and on the distribution of the states of other agents. This leads to control problem in the space of probabilities. First mean field type control problems were considered in \cite{ahmed_ding_controlld}. Note that the mean field type control theory is a counterpart of the mean field game theory proposed in~\cite{Lions01},~\cite{Lions02},~\cite{Huang5}. The  difference between these approaches is that within the framework of mean field games each agent maximizes his/her outcome.

Nowadays, the mean field type control systems are studied primary for the case of dynamics given by SDE (see \cite{Bensoussan_Frehse_Yam_book} and reference therein). Forward-backward stochastic differential equation for mean field type control systems were studied in  \cite{Andersson_Djehiche_2011}, \cite{Buckdahn_Boualem_Li_PMP_SDE}, \cite{Carmona_Delarue_PMP}.  Dynamic programming principle for these systems was also developed (see  \cite{Bayraktar_Cosso_Pham_randomized}, \cite{Bensoussan_Frehse_Yam_equation}, \cite{Lauriere_Pironneau_DPP_MF_control}).  The existence theorem for the optimal control is proved in~\cite{Bahlali_Mezerdi_Mezerdi_existence}. The link between mean field control theory and Hamilton-Jacobi PDEs in the  space of probabilities was studied in~\cite{Pham_Wei_2015_DPP_2016},~\cite{Pham_Wei_2015_Bellman}. Additionally, we mention the papers \cite{Marigonda_Cavagnari}, \cite{Marigonda_et_al_2015}, \cite{Pogodaev} devoted especially to the deterministic mean field type control problem.

Mean field type differential games arise quite naturally when we examine  mean field type control systems assuming that the dynamics of each agents is affected by a disturbances. The pursuit-evasion game where a large group of pursuers tries to chase a large group of evaders can be also treated within the framework of mean field type differential games. Mean field type differential games can be regarded as continuous-time dynamical games in the space of probabilities. Such games appear also in the analysis of  differential games with incomplete information \cite{Cardaliaguet_Quincampoux}, \cite{Cardaliaguet_Souquerie}. In the last case the dynamics of the system does not depend explicitly on the probability distribution.

The  mean field type differential games were previously considered in \cite{Djehiche_Hamdine}. In that paper the  open-loop solution of the differential game was studied for the case when dynamics is given by SDE. 

Aiming to develop the feedback approach to mean field type differential games we assume that the players are informed about the state of the game. Recall that the state of the mean field type differential game is a probability describing the current distribution of all agents. We apply the methodology developed by  Krasovskii and Subbotin  to deterministic mean field type differential games. 
The main idea of this approach is to 
\begin{itemize}
	\item introduce the notions of $u$- and $v$-stability,
	\item given a $u$-stable (respectively, $v$-stable) function, construct a suboptimal strategy of the first (respectively, second) player
	\item prove that there exists a value function which is simultaneously $u$- and $v$-stable.
\end{itemize} 
In the paper to prove the existence theorem we adapt the programmed iteration method first proposed in \cite{Chentsov_Math_sb}, \cite{Chistyakov}, \cite{subb_chen}. Originally, programmed iteration method served as an analytical tool for computing the value function of a finite dimensional differential game.  Note that the close construction were used in \cite{cardal_0}, \cite{cardal_1} where the numerical schemes for differential games were developed.

The paper is organized as follows. In Section \ref{sect:preliminaries} we introduce the  general definitions and notations used throughout the paper. In Section \ref{sect:result} we define the feedback strategies, extend the notion of stability to mean field case and formulate the main results. Section \ref{sect:extremal} is devoted to the construction of suboptimal strategies based on  extremal shift. Finally, in Section \ref{sect:pim} we develop the programmed iteration method  for mean field type differential game and prove the existence theorem for the value function.

\section{General notions and definitions}\label{sect:preliminaries}
If $(X,\rho_X)$ is a separable metric space satisfying the Radon property, then denote by $\mathcal{P}(X)$ the set of probabilities on $X$. Here and below we assume that $(X,\rho_X)$ is endowed by Borel $\sigma$-algebra.  Further, let $\mathcal{P}^2(X)$ be the set of all probabilities $m$ on $X$ such that, for some (and, thus, any) $x_*\in X$,
$$\int_X\rho_X^2(x,x_*)m(dx)<\infty. $$ Denote by $W_2$ the 2-Wasserstein metric on $\mathcal{P}^2(X)$ defined as follows: if $m_1,m_2\in \mathcal{P}^2(X)$, then
\begin{equation*}
W_2(m_1,m_2)\triangleq \left[\inf_{\pi\in \Pi(m_1,m_2)}\int_{X\times X}\rho_X^2(x_1,x_2)\pi(d(x_1,x_2))\right]^{1/2}.
\end{equation*}
Here $\Pi(m_1,m_2)$ denotes the set of plans between $m_1,m_2$, i.e., $\Pi(m_1,m_2)$ is a set of probabilities $\pi\in\mathcal{P}(X\times X)$ such that for any measurable  $\Upsilon\subset X$
$\pi(\Upsilon\times X)=m_1(\Upsilon)$, $\pi(X\times \Upsilon)=m_2(\Upsilon)$. 

Note that if $X$ is Polish, then $\mathcal{P}^2(X)$ is also Polish.  The sets $\mathcal{P}(X)$, $\mathcal{P}^2(X)$ coincide when  $X$ is a compact. Moreover, in this case $\mathcal{P}^2(X)$ is compact and $W_2$ metricizes the narrow convergence \cite{Ambrosio}.

Now, let $(X,\rho_X)$ and $(Y,\rho_Y)$ be separable metric spaces satisfying Radon property. The function $b:X\rightarrow \mathcal{P}(Y)$ is called weakly measurable if, for any $\phi\in C_b(X\times Y)$, the function
$$x\mapsto \int_Y\phi(x,y)b(x,dy) $$ is measurable. Here and below we write $b(x,dy)$ instead of $b(x)(dy)$.
Denote by $\mathrm{WM}(X,Y)$ the space of weakly measurable functions $b$ from $X$ to $\mathcal{P}(Y)$.  The set of usual measurable functions is embedded into $\mathrm{WM}(X,Y)$ in the following way: if $h:X\rightarrow Y$ is measurable, then put
\begin{equation}\label{intro:h_delta}
b_h(x)\triangleq \delta_{h(x)}. 
\end{equation} Here $\delta_y$ stands for the Dirac measure concentrated in $y$. 

If $(Z,\rho_Z)$  is also a  separable space satisfying the Radon property, $\xi\in\mathcal{P}(Y)$, $\zeta\in\mathcal{P}(Z)$ then let $\xi\zeta$ stand for the product of probabilities, i.e., $\xi\zeta$ is a probability on $Y\times Z$ defined by the rule: for $\phi\in C_b(Y\times Z)$
\begin{equation*}\label{intro:product_probabilities}
\int_{Y\times Z}\phi(y,z)(\xi\eta)(d(y,z))\triangleq \int_Y\int_Z\phi(y,z)\xi(dy)\zeta(dz).
\end{equation*}

If $b\in \mathrm{WM}(X,Y)$, $c\in \mathrm{WM}(X,Z)$, then denote by $bc$ the weakly measurable function from $X$ to $\mathcal{P}(Y\times Z)$ given by: for $x\in X$,
\begin{equation}\label{intro:prod_funct}
(bc)(x,d(y,z))\triangleq b(x,dy)c(x,dz).
\end{equation} 

Let $m$ be a finite measure on $X$, $b\in\mathrm{WM}(X,Y)$. Denote by $m\star b$ the probability on $X\times Y$ given by the following rule: for $\phi\in C_b(X\times Y)$,
\begin{equation}\label{intro:product}
\int_{X\times Y}\phi(x,y)(m\star b)(d(y,z)) \triangleq \int_{X}\int_{Y}\phi(x,y)b(x,dy)m(dx).
\end{equation}

Further, denote by $\Lambda(X,m,Y)$ the quotient space of $\mathrm{WM}(X,Y)$ w.r.t. equivalence given by coincidence $m$-a.e., i.e., $b_1$ is equivalent to $b_2$ iff $m\star b_1=m\star b_2$. We say that $\{b_k\}_{k=1}^\infty$  converges narrowly to $b$ if $\{m\star b_k\}$ converges narrowly to $m\star b$, i.e.,  for any $\phi\in C_b(X\times Y)$,
$$\int_X\int_Y\phi(x,y)b_k(x,dy)m(dx)\rightarrow \int_X\int_Y\phi(x,y)b(x,dy)m(dx).$$

Notice, that if $X$ and $Y$ are Polish, then the space $\Lambda(X,Y,m)$ is Polish. The set of measurable functions from $X$ to $Y$ is dense in $\Lambda(X,m,Y)$ \cite{Warga}. Moreover, when $X$ and $Y$ are metric compact,  this property is inherited by $\Lambda(X,m,Y)$. 


Let $\pi$ be a measure on $X\times Y$, then denote by $\pi(\cdot|x)$ (respectively, $\pi(\cdot|y)$) the disintegration of $\pi$ with respect to its marginal on $X$ (respectively, $Y$). We refer to   \cite[III.70]{Delacherie_Meyer} for the the existence result of the disintegration of the probability.

If $(\Omega',\mathcal{F}')$, $(\Omega'',\mathcal{F}'')$ are measurable spaces, $m$ is a probability on $\mathcal{F}'$, $h:\Omega'\rightarrow\Omega''$ is measurable, then denote by $h_\# m$ the push-forward measure: if $\Upsilon\in\mathcal{F}''$,
$$(h_\# m)(\Upsilon)\triangleq m(h^{-1}(\Upsilon)). $$

For simplicity we assume periodic boundary conditions, i.e.,  the phase space is $\td\triangleq \rd\slash\mathbb{Z}^d$. This means that each element $x\in\td$ is an equivalence class $[x']\triangleq \{y'\in\rd:y'\sim x'\}$, where $x'\sim y'$ iff $x'-y'\in\mathbb{Z}^d$. 
The function $$\td\times\td\ni(x,y)\mapsto \|x-y\|\triangleq\min\{\|x'-y'\|:x'\in x,y'\in y\} $$ provides the metric on $\td$.

Denote by $\mathcal{C}$ the space $C([0,T],\td)$.  With some abuse of notation, for $x(\cdot),y(\cdot)\in\mathcal{C}$, we set
$$\|x(\cdot)-y(\cdot)\|\triangleq \max_{t\in [0,T]}\|x(t)-y(t)\|. $$ 
Further, for $t\in [0,T]$, denote by $e_t$ the evaluation operator from $\mathcal{C}$ to $\td$ defined by the following rule: if $x(\cdot)\in \mathcal{C}$, then 
$$e_t(x(\cdot))\triangleq x(t). $$

Since, for every $x(\cdot),y(\cdot)\in\mathcal{C}$, $\|e_t(x(\cdot))-e_t(y(\cdot)))\|\leq \|x(\cdot)-y(\cdot)\| $, we have that if $\chi_1,\chi_2\in\mathcal{P}^2(\mathcal{C})$, $t\in [0,T]$, then
\begin{equation}\label{ineq:wass_motions}
W_2(e_t{}_\#\chi_1,e_t{}_\#\chi_2)\leq W_2(\chi_1,\chi_2). 
\end{equation}

Below we call any function of time taking values in $\mathcal{P}^2(\td)$ a flow of probabilities.
Denote by $\mathcal{M}$ the set of continuous functions from $[0,T]$ to $\mathcal{P}^2(\td)$. Further, let $\mathcal{M}^{R}$ denote the set of flows of probabilities $m(\cdot)\in\mathcal{M}$ such that $$W_2(m(t'),m(t''))\leq R|t'-t''|. $$

Assume that the motion of each agent is given by the ordinary differential equation
\begin{equation}\label{sys:agent}
\begin{split}
\frac{d}{dt}x(t)&=f(t,x(t),m(t),u(t,x(t),m(t)),v(t,x(t),m(t))), \\ &t\in [0,T],\ \ x(t)\in \td,\ \ m(t)\in\mathcal{P}^2(\td), \\ &u(t,x(t),m(t))\in U,\ \  v(t,x(t),m(t))\in V.
\end{split}
\end{equation}

Here $m(t)$ stands for the distribution of all agents;  $U$ (respectively, $V$) denotes the control space of the first (receptively, second) player; $f$ is a function defined on $[0,T]\times\td\times\mathcal{P}^2(\td)\times U\times V$ with values in $\rd$. Integrating formally~(\ref{sys:agent}), we can write down the equation on $m(t)$  in the following form:
\begin{equation*}
\frac{d}{dt}m(t)=\langle f(t,\cdot,m(t),u(t,\cdot,m(t)),v(t,\cdot,m(t))),\nabla\rangle m(t).
\end{equation*} Here $\cdot$ stands for $x$. 

We assume that the controls $u(t,x,m)$ (respectively, $v(t,x,m)$) are chosen by the first (respectively, second) player to minimize (respectively, maximize) the objective function
$$g(m(T)). $$

We assume that
\begin{itemize}
	\item the control sets $U$ and $V$ are metric compacts;
	\item the functions $f$ and $g$ are continuous;
	\item the function $f$ is Lipschitz continuous w.r.t $x$ and $m$;
	\item the Isaacs condition holds true: for any $t\in [0,T]$, $x\in\td$, $m\in \mathcal{P}^2(\td)$, $w\in\rd$:
	\begin{equation*}
	\min_{u\in U}\max_{v\in V}\langle w,f(t,x,m,u,v)\rangle=
	\max_{v\in V}\min_{u\in U}\langle w,f(t,x,m,u,v)\rangle.
	\end{equation*}
\end{itemize}

In the paper we use  relaxed controls.  

Let $\mathcal{U}\triangleq \Lambda([0,T],\lambda,U)$, $\mathcal{V}\triangleq \Lambda([0,T],\lambda,V)$ be space of relaxed controls of the first and second players respectively. Here $\lambda$ stands for the Lebesgue measure on $[0,T]$. Denote by $\mathcal{U}^0$ (respectively, $\mathcal{V}^0$) the set of measurable functions from $[0,T]$ to $U$ (respectively, $V$). Notice that an element of $U$ (respectively $V$) can be regarded as a contant control of the first (respectively, second) player. Using this and the embedding given by (\ref{intro:h_delta}), we  assume that
\begin{equation}\label{inclusion:control_spaces}
U\subset\mathcal{U}^0\subset \mathcal{U},\ \ V\subset\mathcal{V}^0\subset \mathcal{V}. 
\end{equation}

Denote by $\mathcal{W}$ the set of joint relaxed controls of both players on $[0,T]$, i.e.,
$\mathcal{W}\triangleq \Lambda([0,T],\lambda,U\times V)$.  Further, using (\ref{intro:prod_funct}), without loss of generality we assume that
$$\mathcal{U}\times\mathcal{V}\subset\mathcal{W}. $$

If $s\in [0,T]$, $y\in\td$, $m(\cdot)\in \mathcal{M}$, $\eta\in \mathcal{W}$, then denote by $x(\cdot,s,y,m(\cdot),\eta)$ the solution of the initial value problem
\begin{equation}\label{sys:agent_gen}
\frac{d}{dt}x(t)=\int_{U\times V}f(t,x(t),m(t),u,v)\eta(t,d(u,v)),\ \ x(s)=y. 
\end{equation} Note that $x(\cdot,s,y,m(\cdot),\eta)$ is a motion of a representative player produced by the relaxed control of both players $\eta$.

Further, given $s\in [0,T]$, $m(\cdot)\in\mathcal{M}$, denote by $\mathrm{traj}^{s}_{m(\cdot)}:\td\times\mathcal{W}\rightarrow \mathcal{C}$ the operator which assigns to each pair $(y,\eta)$ the motion $x(\cdot,s,y,m(\cdot),\eta)$.

Now let us introduce the sets of distributions of controls. Put 
$$\mathcal{A}\triangleq \mathrm{WM}(\td,\mathcal{U}), \ \ \mathcal{B}\triangleq \mathrm{WM}(\td,\mathcal{V}),\ \ \mathcal{D}\triangleq \mathrm{WM}(\td,\mathcal{W}). $$ The set $\mathcal{A}$ (respectively, $\mathcal{B}$) is the set of distributions of relaxed controls of the first (respectively,  second) player; while $\mathcal{D}$ is the set of distributions of joint relaxed controls of both players.

In some cases we will assume that the players use constant controls. Let $\mathcal{A}^c\triangleq \mathrm{WM}(\td,U)$ denote the set of distributions of constant controls of the first player; and let $\mathcal{B}^c\triangleq \mathrm{WM}(\td,V)$ be the set  distributions of constant controls of the second player.

Due to identification of measurable control with the corresponding relaxed control (see (\ref{inclusion:control_spaces})) we get
$$\mathcal{A}^c\subset\mathcal{A},\ \ \mathcal{B}^c\subset\mathcal{B}. $$

If $\alpha\in\mathcal{A}$, then denote by $\mathcal{D}_1[\alpha]$ the set of distributions $\varkappa\in \mathrm{WM}(\td,\mathcal{U}\times\mathcal{V})$ such that, for any $x\in \td$, the marginal distribution of $\varkappa(x)$ on $\mathcal{U}$ is $\alpha(x)$. Informally speaking, elements of $\mathcal{D}_1[\alpha]$ are distribution of  controls of both players consistent with the distribution of the first player's controls $\alpha$. 
Further, let $\mathcal{D}_1^0[\alpha]$ be the set of $\varkappa\in \mathcal{D}_1[\alpha]$ such that $\varkappa(x)$ is concentrated on $\mathcal{U}\times\mathcal{V}^0$, i.e., in this case we admit only usual controls of the second player.   Analogously, let $\mathcal{D}_2[\beta]$ (respectively, $\mathcal{D}_2^0[\beta]$) denote the set of $\varkappa\in\mathrm{WM}(\td,\mathcal{U}\times\mathcal{V})$ (respectively, $\varkappa\in\mathrm{WM}(\td,\mathcal{U}^0\times\mathcal{V})$) such that, for each $x\in\td$, the marginal distribution of $\varkappa(x)$ on $\mathcal{V}$ is $\beta(x)$. 

\begin{definition}\label{def:motion}	Let $s\in [0,T]$, $m_*\in\mathcal{P}^2(\td)$, $\varkappa\in\mathcal{D}$. We say that $m(\cdot)=m(\cdot,s,m_*,\varkappa)\in \mathcal{M}$ is a flow of probabilities generated by distribution of joint relaxed controls of the players $\varkappa$ if there exists a probability $\gamma\in \mathcal{P}^2(\mathcal{C})$ such that
	$$m(t)=e_t{}_\#\gamma,\ \ m(s)=m_*,$$ and, 
	$$\gamma=\mathrm{traj}^{s}_{m(\cdot)}{}_\#(m_*\star\varkappa). $$
\end{definition}

The existence and uniqueness theorem for the flow of probabilities  is a slight revision of~\cite[Theorem I.1.1]{Sznitman} (see also \cite[Theorem 7.11]{Kol_book}).

\section{Main result}\label{sect:result}

We consider the concept of feedback strategies going back to formalization of zero-sum differential games proposed by Krasovskii and Subbotin. 
In this case the original differential game is replaced with the couple of games, namely, upper and lower games.  In the upper game the first player uses feedback strategy and forms his control stepwise, when the second player forms his control arbitrarily. In the lower game the players change their places. If the upper and lower value functions coincide, then we say that there exists a value function of the original game. 

The strategy of the first player is  a function $\mathfrak{u}:[0,T]\times\mathcal{P}^2(\td)\rightarrow \mathcal{A}^c$. Analogously,  any function $\mathfrak{v}:[0,T]\times\mathcal{P}^2(\td)\rightarrow \mathcal{B}^c$ is called a strategy of the second player. 

Let us start with the upper game. Let $\mathfrak{u}$ be a strategy of the first player, $t_0\in [0,T]$ be an initial time, $m_0\in\mathcal{P}^2(\td)$ be an initial distribution of players, and let $\Delta=\{t_i\}_{i=0}^N$ be a partition of $[t_0,T]$. 

\begin{definition}
	We say that a flow of probabilities  $m(\cdot):[t_0,T]\rightarrow \mathcal{P}^2(\td)$ is  generated by $t_0$, $m_0$, $\mathfrak{u}$ and $\Delta$ if $m(t_0)=m_0$ and there exist distributions of controls $\varkappa_i\in\mathcal{D}^0_1[\mathfrak{u}[t_i,m(t_i)]]$, $i=0,\ldots,N-1$, such that, for $t\in [t_i,t_{i+1}]$, $i=0,\ldots, N-1$, 
	$$m(t)=m(t,t_i,m(t_i),\varkappa_i).$$

	We denote the set of flows of probabilities generated by  $t_0$, $m_0$, $\mathfrak{u}$, and $\Delta$ by $\mathcal{X}_1(t_0,m_0,\mathfrak{u},\Delta)$.
\end{definition} If the first player uses the strategy $\mathfrak{u}$ and corrects his/her control at times of partition $\Delta$, then his/her outcome is evaluated by the value 
$$J_1(t_0,m_0,\mathfrak{u},\Delta) \triangleq \sup\{g(m(T)):m(\cdot)\in \mathcal{X}_1(t_0,m_0,\mathfrak{u},\Delta)\}.$$

Analogously, if $t_0$ is an initial time, $m_0$ is an initial distribution of players, $\mathfrak{v}$ is a second player's strategy, $\Delta$ is a partition of the time interval $[t_0,T]$, one can introduce the  set of flows of probabilities generated by $t_0$, $m_0$, $\Delta$ and $\mathfrak{v}$. Denote it by 
$\mathcal{X}_2(t_0,m_0,\mathfrak{v},\Delta)$.
The value 
$$J_2(t_0,m_0,\mathfrak{v},\Delta) \triangleq \inf\{g(m(T)):m(\cdot)\in \mathcal{X}_2(t_0,m_0,\mathfrak{v},\Delta)\}$$ provides an evaluation of the outcome when the system starts at $(t_0,m_0)$ and the second player chooses the strategy $\mathfrak{v}$ and the partition $\Delta$.

The upper value of the game at $(t_0,m_0)$ is equal to
$$\Gamma_1(t_0,m_0)\triangleq \inf_{\mathfrak{u},\Delta}J_1(t_0,m_0,\mathfrak{u},\Delta). $$ The lower value of the game is defined in the same way:
$$\Gamma_2(t_0,m_0)\triangleq \sup_{\mathfrak{v},\Delta}J_2(t_0,m_0,\mathfrak{v},\Delta). $$
Clearly,
\begin{equation*}\label{ineq:gammas}
\Gamma_1\geq \Gamma_2. 
\end{equation*}
If $\Gamma_1=\Gamma_2$, then we say that the game has a value.

To formulate the main result we require a notions of $u$- and $v$- stability. 
\begin{definition}\label{def:u_stable} We say that a lower semicontinuous function $\psi_1:[0,T]\times \mathcal{P}^2(\td)\rightarrow\mathbb{R}$ is  $u$-stable if
	\begin{itemize}
		\item for any $m\in\mathcal{P}^2(\td)$, $g(m)\leq \psi_1(T,m)$;
		\item for any $s,r\in [0,T]$, $s<r$, $m_*\in\mathcal{P}^2(\td)$, $\beta\in\mathcal{B}^c$, there exists a distribution of  players' controls $\varkappa\in \mathcal{D}_2[\beta]$ such that
		$$\psi_1(s,m_*)\geq \psi_1(r,m(r,s,m_*,\varkappa)). $$ 
	\end{itemize}
\end{definition}

\begin{definition}\label{def:v_stable} An upper semicontinuous function $\psi_2:[0,T]\times \mathcal{P}^2(\td)\rightarrow\mathbb{R}$ is called $v$-stable if 
	\begin{itemize}
		\item for any $m\in\mathcal{P}^2(\td)$, $g(m)\geq \psi_2(T,m)$
		\item
		for any $s,r\in [0,T]$, $s<r$, $m_*\in\mathcal{P}^2(\td)$, $\alpha\in\mathcal{A}^c$, there exists a distribution of players' controls $\varkappa\in \mathcal{D}_1[\alpha]$ such that
		$$\psi_2(s,m_*)\leq \psi_2(r,m(r,s,m_*,\varkappa)).$$
	\end{itemize}
\end{definition}

\begin{theorem}\label{th:extremal}
	Let $\psi_1$ be a $u$-stable function, then, for any $(t_0,m_0)\in [0,T]\times\mathcal{P}^2(\td)$,
	$$\Gamma_1(t_0,m_0)\leq \psi_1(t_0,m_0). $$  If $\psi_2$ is a $v$-stable function, then 
	$$\Gamma_2(t_0,m_0)\geq \psi_2(t_0,m_0). $$
\end{theorem}

One way to prove the existence and compute the value function is the programmed iteration method first proposed for zero-sum differential game by Chentsov \cite{Chentsov_Math_sb}. 

Let $\mathbb{R}^{[0,T]\times\mathcal{P}^2(\td)}$ denote the set of all functions from $[0,T]\times\mathcal{P}^2(\td)$ to $\mathbb{R}$.

Define the operator $\Phi:\mathbb{R}^{[0,T]\times\mathcal{P}^2(\td)}\rightarrow \mathbb{R}^{[0,T]\times\mathcal{P}^2(\td)}$   by the following rule: if $\omega\in \mathbb{R}^{[0,T]\times\mathcal{P}^2(\td)}$, $s\in [0,T]$, $\mu\in \mathcal{P}^2(\td)$, then
$$\Phi[\omega](s,\mu)\triangleq \sup_{r\in [s,T]}\sup_{\beta\in\mathcal{B}^c} \inf_{\varkappa\in\mathcal{D}_2[\beta]}\omega(r,m(r,s,\mu,\varkappa)). $$

Let \begin{equation}\label{intro:omega_null}
\omega_0(s,\mu)\triangleq \sup_{\beta\in \mathcal{B}^c}\inf_{\varkappa\in\mathcal{D}_2[\beta]} g(m(T,s,\mu,\varkappa)). 
\end{equation}
Further, for $k=1,2,\ldots$, set
\begin{equation}\label{intro:omega_k}
\omega_k\triangleq \Phi[\omega_{k-1}]. 
\end{equation}

Analogously, let us introduce the operator $\Psi:\mathbb{R}^{[0,T]\times\mathcal{P}^2(\td)}\rightarrow\mathbb{R}^{[0,T]\times\mathcal{P}^2(\td)}$  by
$$\Psi[\omega](s,\mu)\triangleq \inf_{r\in [s,T]}\inf_{\alpha\in\mathcal{A}^c} \sup_{\varkappa\in\mathcal{D}_1[\alpha]}\omega(r,m(r,s,\mu,\varkappa)). $$ 
Put
\begin{equation}\label{intro:omega_v_null}
\omega^0(s,\mu)\triangleq \inf_{\alpha\in \mathcal{A}^c}\sup_{\varkappa\in\mathcal{D}_1[\alpha]} g(m(T,s,\mu,\varkappa)), 
\end{equation}
\begin{equation}\label{intro:omega_v_k}
\omega^k\triangleq \Psi[\omega^{k-1}]. 
\end{equation}

\begin{theorem}\label{th:existence}
	There exists a value function of the mean field type differential game~$\Gamma$.  It  is simultaneously $u$- and $v$-stable. For any $s\in [0,T]$, $\mu\in\mathcal{P}^2(\td)$,
	$$\Gamma(s,\mu)=\lim_{k\rightarrow\infty}\omega_{k}(s,\mu)=\lim_{k\rightarrow\infty}\omega^{k}(s,\mu). $$
\end{theorem}

\section{Extremal shift rule}\label{sect:extremal}
In this section we prove Theorem \ref{th:extremal}. To this end, given the $u$-stable function $\psi_1$ and $\varepsilon>0$, we construct the strategy $\mathfrak{u}_\varepsilon$ such that, for a sufficiently fine partition, the corresponding outcome is estimated by $\psi_1$ with an error vanishing when $\varepsilon\rightarrow 0$.

Since the function $f$ is continuous, the sets $\td$, $\mathcal{P}^2(\td)$, $U$, $V$ are compact, there exists a constant  $C_0$ such that, for all $t\in [0,T]$, $x\in \td$, $m\in\mathcal{P}^2(\td)$, $u\in U$, $v\in V$, \begin{equation}\label{intro:C_0}
\|f(t,x,m,u,v)\|\leq C_0.
\end{equation} 
Therefore, for all $s\in [0,T]$, $m_*\in\mathcal{P}^2(\td)$, $\varkappa\in \mathcal{D}$, $m(\cdot,s,m_*,\varkappa)\in \mathcal{M}^{C_0}$. 

Let the  functions $\varpi_f,\varpi_g:\mathbb{R}\rightarrow [0,+\infty)$ be vanishing  at $0$, continuous at~$0$ and satisfy
$$\|f(t,x,m,u,v)-f(t',x,m,u,v)\|\leq \varpi_f(t-t'), $$
$$|g(m)-g(m')|\leq \varpi_g(W_2(m,m')) $$  for any $t,t'\in [0,T]$, $x\in\td$, $m,m'\in\mathcal{P}^2(\td)$, $u\in U$, $v\in V$. Without loss of generality, one can assume that $\varpi_f$, $\varpi_g$ are nondecreasing on $[0,+\infty)$. Moreover, we assume that $\varpi_f$ is even. Further, denote by $L$ the Lipschitz constant for the function $f$, i.e., for any $t\in [0,T]$, $x,x'\in\td$, $m,m'\in\mathcal{P}^2(\td)$, $u\in U$, $v\in V$,
$$\|f(t,x,m,u,v)-f(t,x',m',u,v)\|\leq L\|x-x'\|+LW_2(m,m'). $$ 

Set \begin{equation}\label{intro:varpi_1}
\varpi_1(\varepsilon)\triangleq 2\sqrt{d}\varpi_f(\varepsilon)+4\sqrt{d}LC_0\varepsilon,
\end{equation}
\begin{equation}\label{intro:varpi_2}
\varpi_2(\varepsilon)\triangleq 2\varpi_1(\varepsilon)+4C_0^2\varepsilon.
\end{equation}
Let $\rho(\varepsilon,t)$ be equal to 
\begin{equation}\label{intro:varrho}
\varrho(\varepsilon,t)\triangleq (\varepsilon+\varpi_2(\varepsilon)t)e^{4Lt}. 
\end{equation}

Given $s\in [0,T]$, $x,y\in\td$, $m\in\mathcal{P}^2(\td)$, pick
$$ \hat{u}(s,x,y,m)\in\underset{u\in U}{\operatorname{argmin}}\max_{v\in V}\langle x'-y',f(s,x,m,u,v)\rangle, $$
$$ \hat{v}(s,x,y,m)\in \underset{v\in V}{\operatorname{argmax}}\min_{u\in U}\langle x'-y',f(s,x,m,u,v)\rangle. $$ Here $x'\in x$, $y'\in y$ are such that $\|x'-y'\|=\|x-y\|$. 
Notice that one can choose the function $\hat{u}$ and $\hat{v}$ to be measurable.

The strategy $\mathfrak{u}_\varepsilon$ introduced below realizes the extremal shift rule initially proposed for finite dimensional differential games in~\cite{NN_PDG_en}. We adapt for the mean field case the variant of this method borrowed from~\cite{krasovskii_control}. Let $(s,m)$ be a position from $[0,T]\times\mathcal{P}^2(\td)$. Let $\nu\in\mathcal{P}^2(\td)$ be  such that $W_2^2(m,\nu)\leq\varrho(\varepsilon,s)$ and
$$\psi_1(s,\nu)=\min\{\psi_1(s,m'):m'\in\mathcal{P}^2(\td),\ \ W_2^2(m,m')\leq \varrho(\varepsilon,s)\}. $$ Now, let $\pi$ be an optimal plan between $m$ and $\nu$, $\pi(\cdot|x)$ be its disintegration with respect to $m$. Define the first player's strategy $\mathfrak{u}_\varepsilon $ by the rule: for $s\in [0,T]$, $m\in\mathcal{P}^2(\td)$, $x\in\td$, put
\begin{equation}\label{intro:strategy}
\mathfrak{u}_\varepsilon[s,m](x)\triangleq \hat{u}(s,x,\cdot,m)_\#\pi(\cdot|x).  
\end{equation} By construction we have that $\mathfrak{u}_\varepsilon[s,m]\in \mathcal{A}^c$.

Below we use the following unfolding of the solution of (\ref{sys:agent_gen}). If $t\in [0,T]$, $x\in\rd$, $m\in\mathcal{P}^2(\td)$, $u\in U$, $v\in V$, then put
$$\tilde{f}(t,x',m,u,v)\triangleq f(t,[x'],m,u,v). $$ 

Let $s\in [0,T]$, $y'\in\rd$, $m(\cdot)\in\mathcal{M}$, $\eta\in\mathcal{W}$. Denote by $\tilde{x}(\cdot,s,y',m(\cdot),\eta)$ the solution of initial value problem in $\rd$
\begin{equation}\label{sys:agent_gen_rd}
\frac{d}{dt}\tilde{x}(t)=\int_{U\times V}\tilde{f}(t,\tilde{x}(t),m(t),u,v)\eta(d(u,v)),\ \ \tilde{x}(s)=y'.
\end{equation}   Notice that if $\tilde{x}(\cdot)$ solves (\ref{sys:agent_gen_rd}), then $x(\cdot)$ given by $x(t)\triangleq [\tilde{x}(t)]$ solves (\ref{sys:agent_gen}) for $y=[y']$. Furthermore, the definition of $\tilde{f}$ and (\ref{intro:C_0}) yield
\begin{equation}\label{ineq:lip_trajectory}
\|\tilde{x}(t,s,y',m(\cdot),\eta)-y'\|\leq C_0(t-s).
\end{equation}


\begin{lemma}\label{lm:agent_motion}
	Let $s,r\in [0,T]$, $s\leq r$, $x_*,y_*\in\td$, $m(\cdot),\nu(\cdot)\in \mathcal{M}^{C_0}$, ${u}_*=\hat{u}(s,x_*,y_*,m(s))$, $v^*=\hat{v}(s,x_*,y_*,m(s))$, $\xi\in\mathcal{U}$, $\zeta\in\mathcal{V}$,  $x(\cdot)=x(\cdot,s,x_*,m(\cdot),\delta_{u_*}\zeta)$, $y(\cdot)=x(\cdot,s,y_*,\nu(\cdot),\xi\delta_{v^*})$.  Then 
	\begin{equation*}
	\begin{split}
	\|x(r)-y(r)\|^2\leq \|x_*&-y_*\|^2(1+3L(r-s))\\&+LW_2^2(m(s),\nu(s))(r-s)+\varpi_2(r-s)\cdot (r-s).
	\end{split}
	\end{equation*}
	Here $\varpi_2$ is introduced by (\ref{intro:varpi_2}).
\end{lemma}
\begin{proof} Pick $x_*'\in x_*$ and $y_*'\in y_*$ such that
	$$\|x_*-y_*\|=\|x_*'-y_*'\| $$ and
	$$ \hat{u}(s,x_*,y_*,m)\in\underset{u\in U}{\operatorname{argmin}}\max_{v\in V}\langle x'_*-y'_*,f(s,x_*,m,u,v)\rangle, $$
	$$ \hat{v}(s,x_*,y_*,m)\in \underset{v\in V}{\operatorname{argmax}}\min_{u\in U}\langle x_*'-y'_*,f(s,x_*,m,u,v)\rangle. $$
	Denote
	$$\tilde{x}(t)\triangleq \tilde{x}(t,s,x_*',m(\cdot),\delta_{u_*}\zeta),\ \ \tilde{y}(t)\triangleq \tilde{x}(t,s,y_*',\nu(\cdot),\xi\delta_{v^*}). $$
	Using (\ref{ineq:lip_trajectory}), we get
	\begin{equation}\label{ineq:x_r_y_r_firts}
	\begin{split}
	\|x(r)-y&(r)\|^2\leq \|\tilde{x}(r)-\tilde{y}(r)\|^2 \\\leq \|x_*'&-y_*'\|^2+\|\tilde{x}(r)-x_*'\|^2+\|\tilde{y}(r)-y_*'\|^2\\&- 2\langle \tilde{x}(r)-x_*',\tilde{y}(r)-y_*'\rangle\\
	&+
	2\langle x_*'-y_*', \tilde{x}(r)-x_*'\rangle-2\langle x_*'-y_*',\tilde{y}(r)-y_*'\rangle\\ \leq \|x_*'&-y_*'\|^2+4C_0^2(r-s)^2\\
	&+
	2\langle x_*'-y_*', \tilde{x}(r)-x_*\rangle-2\langle x_*'-y_*',\tilde{y}(r)-y_*'\rangle.
	\end{split}
	\end{equation}
	
	Using  Lipschitz continuity of the function $f$ w.r.t. $x$ and $m$, the definition of $\tilde{f}$,  estimate (\ref{ineq:lip_trajectory}),  the inequality $\|x_*'-y_*'\|\leq\sqrt{d}$ and the fact that $m(\cdot),\nu(\cdot)\in\mathcal{M}^{C_0}$, we conclude that
	\begin{equation}\label{ineq:inner_product_estimate}
	\begin{split}
	\langle x_*'-y_*', \tilde{x}(r)&-x_*'\rangle-\langle x_*'-y_*',\tilde{y}(r)-y_*'\rangle\\=
	\Bigl\langle x_*'-y_*'&, \int_s^r\int_V \tilde{f}(t,\tilde{x}(t),m(t),u_*,v)\zeta(t,dv)dt\Bigr\rangle\\-
	\Bigl\langle x_*'&-y_*', \int_s^r\int_U \tilde{f}(t,\tilde{y}(t),\nu(t),u,v^*)\xi(t,du)dt\Bigr\rangle\\\leq
	\int_s^r\int_V\langle x&{}_*'-y_*',f(s,x_*,m(s),u_*,v)\zeta(t,dv)\rangle dt\\-
	\int_s^r&\int_U\langle x_*'-y_*',f(s,y_*,\nu(s),u,v^*)\xi(t,du)\rangle dt\\
	+\varpi_1(&r-s)\cdot (r-s).
	\end{split}
	\end{equation}
	Here $\varpi_1$ is defined by (\ref{intro:varpi_1})

	For each  $u\in U$, $v\in V$, the following inequality holds:
	\begin{multline*}
	\langle x_*'-y_*',f(s,x_*,m(s),u_*,v)\rangle-\langle x_*'-y_*',f(s,y_*,\nu(s),u,v^*)\rangle \\ \leq
	\langle x_*'-y_*',f(s,x_*,m(s),u_*,v)\rangle-\langle x_*'-y_*',f(s,x_*,m(s),u,v^*)\rangle\\+\frac{3L}{2}\|x_*-y_*\|^2+ \frac{L}{2}W_2^2(m(s),\nu(s)).
	\end{multline*}
	Using the choice of $u_*$ and $v^*$, we get
	\begin{multline*}
	\langle x_*'-y_*',f(s,x_*,m(s),u_*,v)\rangle-\langle x_*'-y_*',f(s,y_*,\nu(s),u,v^*)\rangle \\ \leq \frac{3L}{2}\|x_*-y_*\|^2+ \frac{L}{2}W_2^2(m(s),\nu(s)).
	\end{multline*}
	Combining this, (\ref{ineq:x_r_y_r_firts}), (\ref{ineq:inner_product_estimate}) and definition of $\varpi_2$ (see (\ref{intro:varpi_2})), we get the conclusion of the Lemma. 
\end{proof}
\begin{lemma}\label{lm:flow_motion} Let $s,r\in [0,T]$, $s\leq r$, $m_*,\nu_*\in\mathcal{P}^2(\td)$, $\pi$ be an optimal plan between $m_*$ and $\nu_*$, $\pi(\cdot|x)$, $\pi(\cdot|y)$ be its disintegration with respect to $m_*$ and $\nu_*$ respectively, $\alpha^*(x)\triangleq \hat{u}(s,x,\cdot,m_*)_\#\pi(\cdot|x)$, $\beta^*(y)\triangleq \hat{v}(s,\cdot,y,m_*)_\#\pi(\cdot|y)$, $\varkappa\in\mathcal{D}_1[\alpha^*]$, $\vartheta\in\mathcal{D}_2[\beta^*]$, $m(\cdot)=m(\cdot,s,m_*,\varkappa)$, $\nu(\cdot)=m(\cdot,s,\nu_*,\vartheta)$. Then
	$$W_2^2(m(r),\nu(r))\leq W_2^2(m_*,\nu_*)(1+4L(r-s))+\varpi_2(r-s)\cdot (r-s). $$
\end{lemma}
\begin{proof} 
	
	Since $\alpha^*\in\mathcal{A}^c=\mathrm{WM}(\td,U)$, $\varkappa\in \mathcal{D}_1[\alpha^*]$, we have that, for each $x\in \td$, there exists a disintegration of $\varkappa(x)$ with respect to $\alpha^*(x)$ that is an element of $\mathrm{WM}( U,\mathcal{V})$. With some abuse of notation denote it by  $\varkappa(x,u)$. Analogously, let $\vartheta(x,v)$ stand for the disintegration of $\vartheta(x)$ with respect to $\beta^*(x)$.
	
	Further, set
	\begin{equation}\label{intro:hat_varkappa}
	\hat{\varkappa}(x_*,y_*)\triangleq\varkappa(x_*,\hat{u}(s,x_*,y_*,m_*)),
	\end{equation}
	\begin{equation}\label{intro:hat_vartheta}
	\hat{\vartheta}(x_*,y_*)\triangleq\vartheta(x_*,\hat{v}(s,x_*,y_*,m_*)).
	\end{equation} Notice that, for each $x_*,y_*\in\td$, $\hat{\varkappa}(x_*,y_*)\in \mathcal{P}(\mathcal{V})$, $\hat{\vartheta}(x_*,y_*)\in\mathcal{P}(\mathcal{U})$.
	
	Let $\Xi$ be a probability on $\td\times\td\times\mathcal{U}\times\mathcal{V}\times\mathcal{U}\times\mathcal{V}$ defined by the following rule: for $\phi\in C_b(\td\times\td\times\mathcal{U}\times\mathcal{V}\times\mathcal{U}\times\mathcal{V})$,
	\begin{equation}\label{intro:Prob_Xi}
	\begin{split}
	&\int_{\td\times\td\times\mathcal{U}\times\mathcal{V}\times\mathcal{U}\times\mathcal{V}}\phi(x,y,\xi',\zeta',\xi'',\zeta'')\Xi(d(x,y,\xi',\zeta',\xi'',\zeta''))\\ &{}\hspace{10pt}\triangleq
	\int_{\td\times\td}\int_{\mathcal{U}}\int_{\mathcal{V}}\phi(x_*,y_*,\delta_{\hat{u}(s,x_*,y_*,m_*)},\zeta',\xi'',\delta_{\hat{v}(s,x_*,y_*,m_*)})\\
	&{}\hspace{115pt}\hat{\varkappa}(x_*,y_*,d\zeta')\hat{\vartheta}(x_*,y_*,d\xi'')\pi(d(x_*,y_*)).
	\end{split}
	\end{equation} Obviously, the marginal distribution of $\Xi$ on $\td\times\td$ is $\pi$.
	
	To clarify the link between the probability $\Xi$ and the  distributions of controls $\varkappa$ and $\vartheta$ let us introduce the following projection of $\td\times\td\times\mathcal{U}\times\mathcal{V}\times\mathcal{U}\times\mathcal{V}$ on $\td\times \mathcal{U}\times\mathcal{V}$:
	\begin{itemize}
		\item $\mathrm{P}^{1}(x_*,y_*,\xi',\zeta',\xi'',\zeta'')\triangleq (x_*,\xi',\zeta')$;
		\item $\mathrm{P}^{2}(x_*,y_*,\xi',\zeta',\xi'',\zeta'')\triangleq (y_*,\xi'',\zeta'')$.
	\end{itemize} We have that
	\begin{equation}\label{equality:projections}
	\mathrm{P}^1{}_\#\Xi=m_*\star \varkappa,\ \ \mathrm{P}^2{}_\#\Xi=\nu_*\star \vartheta. 
	\end{equation}
	
	Let $\mathcal{T}^{r,s}_{m(\cdot),\nu(\cdot)}$ be the operator from $\td\times \td\times \mathcal{U}\times\mathcal{V}\times\mathcal{U}\times\mathcal{V}$ to $\td\times\td$ defined by:
	\begin{equation}\label{intro:T}
	\begin{split}
	\mathcal{T}^{r,s}_{m(\cdot),\nu(\cdot)}&(x',x'',\xi',\zeta',\xi'',\zeta'')\\&\triangleq (e_r(\mathrm{traj}^{s}_{m(\cdot)}(x',\xi'\zeta')), e_r(\mathrm{traj}^{s}_{\nu(\cdot)}(x'',\xi''\zeta''))).  
	\end{split}
	\end{equation}
	
	This, (\ref{equality:projections}) and Definition \ref{def:motion} yield that
	$$\hat{\pi}\triangleq \mathcal{T}^{r,s}_{m(\cdot),\nu(\cdot)}{}_\# \Xi $$ is a plan between $m(r)$ and $\nu(r)$. To simplify notation put
	\begin{equation}\label{intro:hat_x}
	\hat{x}(x_*,y_*,\zeta)\triangleq x(r,s,x_*,m(\cdot),\delta_{\hat{u}(s,x_*,y_*,m_*)}\zeta), 
	\end{equation}
	\begin{equation}\label{intro:hat_y}
	\hat{y}(x_*,y_*,\xi)\triangleq x(r,s,y_*,\nu(\cdot),\xi\delta_{\hat{v}(s,x_*,y_*,m_*)}).
	\end{equation}
	
	Using (\ref{intro:Prob_Xi}), (\ref{intro:T}) and definition of $\hat{\pi}$, we get
	\begin{equation*}
	\begin{split}
	W_2^2(m(r),\nu(r))
	\leq \int_{\td\times\td}&\|x-y\|^2\hat{\pi}(d(x,y))\\=
	\int_{\td\times\td\times\mathcal{U}\times\mathcal{V}\times\mathcal{U}\times\mathcal{V}}&\|x(r,s,x_*,m(\cdot),\xi'\zeta')-x(r,s,y_*,\nu(\cdot),\xi''\zeta'')\|^2 \\ 
	&\Xi(d(x_*,y_*,\xi',\zeta',\xi'',\zeta''))
	\\=
	\int_{\td\times\td}\int_{\mathcal{U}}\int_{\mathcal{V}}
	\|\hat{x}&(x_*,y_*,\zeta)-\hat{y}(x_*,y_*,\xi)\|^2\\ &\hat{\varkappa}(x_*,y_*,d\zeta)\hat{\vartheta}(x_*,y_*,d\xi)\pi(d(x_*,y_*)).
	\end{split}
	\end{equation*} Since $m(\cdot),\nu(\cdot)\in\mathcal{M}^{C_0}$, taking into account designations (\ref{intro:hat_x}), (\ref{intro:hat_y}), one can estimate $\|\hat{x}(x_*,y_*,\zeta)-\hat{y}(x_*,y_*,\xi)\|^2$ by Lemma \ref{lm:agent_motion}. This implies the inequality
	\begin{equation*}\begin{split}
	W_2^2(m(r),\nu(r))\leq (1+&3L(r-s))\int_{\td\times\td}\|x_*-y_*\|^2\pi(d(x_*,y_*))\\&+LW_2^2(m(s),\nu(s))(r-s)+\varpi_2(r-s)\cdot (r-s).
	\end{split}
	\end{equation*} Since $\pi$ is an optimal plan between $m_*$ and $\nu_*$ we obtain the conclusion of the Lemma. 
\end{proof}

\begin{proof}[Proof of Theorem   \ref{th:extremal}]
Let $t_0\in [0,T]$, $m_0\in\mathcal{P}^2(\td)$, $\Delta=\{t_i\}_{i=0}^N$ be a partition of $[t_0,T]$. Assume that $d(\Delta)\leq\varepsilon$. As usual, $d(\Delta)$ stands for the fineness of $\Delta$.

Let $m(\cdot)\in \mathcal{X}_1(t_0,m_0,\mathfrak{u}_\varepsilon,\Delta)$. Denote $m_i\triangleq m(t_i)$. Recall (see (\ref{intro:strategy})) that the strategy $\mathfrak{u}_\varepsilon[t_i,m_i]$ is determined by the rule: if $x\in\td$,
$$\mathfrak{u}_\varepsilon[t_i,m_i](x)=\alpha_i^*(x,\cdot)=\hat{u}(t_i,x,\cdot,m_i)_\#\pi_i(\cdot|x). $$ Here $\pi_i$ is an optimal plan between $m_i$ and $\nu_i$; while $\nu_i$ is a probability on $\mathcal{P}^2(\td)$ such that 
$$\nu_i\in\underset{m:W_2^2(m_i,m)\leq \varrho(\varepsilon,t_i)}{\operatorname{argmin}}\psi_1(t_i,m). $$   Since $m(\cdot)\in \mathcal{X}_1(t_0,m_0,\mathfrak{u}_\varepsilon,\Delta)$, there exists a sequence of distributions $\varkappa_i\in \mathcal{D}$, $i=0,\ldots,N-1$ such that $\varkappa_i\in\mathcal{D}_1^0[\alpha_i^*]$ and, for $t\in [t_i,t_{i+1}]$,
$$m(t)=m(t,t_i,m_i,\varkappa_i). $$
Further, for $y\in\td$, set
$$\beta^*_i(y,\cdot)=\hat{v}(t_i,\cdot,y,m_i)_\#\pi_i(\cdot|y).  $$ We have that there exists a distribution of controls $\vartheta_i\in\mathcal{D}_2[\beta_i^*]$ such that, for $\nu_i(\cdot)\triangleq m(\cdot,t_i,\nu_i,\vartheta_i)$,
\begin{equation}\label{ineq:u_stab_i}
\psi_1(t_i,\nu_i)\geq \psi_1(t_{i+1},\nu(t_{i+1})).
\end{equation} Using Lemma \ref{lm:flow_motion}, we get 
\begin{equation}\label{ineq:distance_m_nu}
W_2^2(m(t_{i+1}),\nu_i(t_{i+1}))\leq \varrho(\varepsilon,t_{i+1}). 
\end{equation}
From (\ref{ineq:u_stab_i}) and Definition \ref{def:u_stable} we conclude  that
\begin{equation*}\begin{split}
\psi_1(t_0,m_0)\geq \psi_1(t_0,\nu_0)&\geq \psi_1(t_1,\nu_0(t_1))\geq \psi_1(t_1,\nu_1)\\ &\geq\ldots\geq\psi_1(t_N,\nu_{N-1}(t_N))\geq g(\nu_{N-1}(t_N)). \end{split}
\end{equation*}

Using  (\ref{ineq:distance_m_nu}) and the definition of $\varrho$ (see (\ref{intro:varrho})), we obtain the following estimate: 
\begin{equation*}\begin{split}g(m(T))\leq g(\nu_{N-1}(T))+\varpi_g(&W_2(m(T),\nu_{N-1}(T)))\\ &\leq
\psi_1(t_0,m_0)+\varpi_g(\varrho^{1/2}(\varepsilon,T)).
\end{split}\end{equation*}
Thus, for any partition $\Delta$ such that $d(\Delta)\leq\varepsilon$,
$$J_1(t_0,m_0,\mathfrak{u}_\varepsilon,\Delta)\leq\psi_1(t_0,m_0)+\varpi_g(\varrho^{1/2}(\varepsilon,T)). $$
Since $\varpi_g(\varrho^{1/2}(\varepsilon,T))\rightarrow 0$ as $\varepsilon\rightarrow 0$, we get 
$$\Gamma_1(t_0,m_0)\leq \psi_1(t_0,m_0). $$
This proves the first statement of the theorem. To prove the second part it suffices to replace $g$ with $-g$ and interchange the players.
\end{proof}

\section{Programmed iteration method}\label{sect:pim}
This section is concerned with the proof of Theorem \ref{th:existence}.
First, let us prove the following auxiliary statement. 
\begin{lemma}\label{lm:cont_dependence}
	Let $\{s_i\}\subset [0,T]$, $\{\mu_i\}\in\mathcal{P}^2(\td)$, $s\in [0,T]$, $\mu\in \mathcal{P}^2(\td)$, $\beta\in \mathcal{B}^c$. Assume that $s_i\rightarrow s$, $W_2(\mu_i,\mu)\rightarrow 0$ as $i\rightarrow \infty$. Then there exists a sequence $\{\beta_i\}_{i=1}^\infty\subset \mathcal{B}^c$ satisfying the following properties:
	\begin{enumerate}
		\item $\{\mu_i\star\beta_i\}$  converges narrowly to $\mu\star\beta$;
		\item if $\varkappa_i\in \mathcal{D}_2[\beta_i]$, then there exist $\varkappa\in\mathcal{D}_2[\beta]$ and a subsequence $\{i_l\}$ such that, for $m_{i_l}(\cdot)\triangleq m(\cdot,s_{i_l},\mu_{i_l},\varkappa_{i_l})$ and $m(\cdot)\triangleq m(\cdot,s,\mu,\varkappa)$,
		$$\lim_{l\rightarrow\infty}\sup_{t\in [0,T]}W_2(m_{i_l}(t),m(t))=0. $$
	\end{enumerate} Moreover, if $\mu_i=\mu$, then one can choose $\beta_i=\beta$.
\end{lemma}
\begin{proof} First, let us construct $\beta_i$. Let $\pi_i\in \Pi(\mu,\mu_i)$ be an optimal plan between $\mu$ and $\mu_i$. Further, let $\pi_i(\cdot|x')$ be its disintegration with respect to $\mu_i$. Define $\beta_i$ by the following rule: for $x'\in\td$ and $\phi\in C_b(V)$, put
	\begin{equation}\label{intro:beta_i}
	\int_V \phi(v)\beta_i(x',dv)\triangleq \int_{\td}\int_V\phi(v)\beta(x,dv)\pi_i(dx|x'). 
	\end{equation}
	
	Let us prove that $\mu_i\star\beta_i$ converges narrowly to $\mu\star\beta$. It suffices to prove that if $\phi\in C_b(\td\times V)$ is 1-Lipschitz continuous, then
	$$\int_{\td\times V}\phi(x',v)(\mu_i\star\beta_i)(d(x',v))\rightarrow \int_{\td\times V}\phi(x,v)(\mu\star \beta)(d(x,v)). $$ Since $\pi_i$ is an optimal plan between $\mu$ and $\mu_i$, by (\ref{intro:beta_i}) we have that
	\begin{equation}
	\begin{split}
	\Bigl| \int_{\td\times V}&\phi(x',v)(\mu_i\star\beta_i)(d(x',v))-\int_{\td\times V}\phi(x,v)(\mu\star \beta)(d(x,v))\Bigr| \\ &=\Bigl|
	\int_{\td\times\td}\int_V \phi(x',v)\beta(x,dv)\pi_i(d(x,x'))\\ &{}\hspace{40pt} -\int_{\td\times\td}\int_V \phi(x,v)\beta(x,dv)\pi_i(d(x,x'))
	\Bigr| \\ &\leq
	\int_{\td\times\td}\int_V|\phi(x',v)-\phi(x,v)|\beta(x,dv)\pi_i(d(x,x'))\\ &\leq
	\int_{\td\times\td}\|x-x'\|\pi_i(d(x,x'))\leq W_2(\mu,\mu_i).
	\end{split}
	\end{equation} This proves the first statement of the lemma. By construction we have that $\beta_i=\beta$ when $\mu_i=\mu$. 
	
	Now, let $\varkappa_i\in\mathcal{D}_2[\beta_i]$. We have that $\tau_i\triangleq \mu_i\star\varkappa_i$ is a probability on compact set $\td\times\mathcal{U}\times V$. Thus, there exist a subsequence $\{i_l\}$ and a probability $\tau\in\mathcal{P}(\td\times\mathcal{U}\times V)$ such that $\{\tau_{i_l}\}$ converges narrowly to $\tau$. Since the marginal distribution of $\tau_i$ on $\td\times V$ is $\mu_i\star\beta_i$ and $\mu_i\star\beta_i\rightarrow \mu\star\beta$ we conclude  that the marginal distribution of $\tau$ on $\td\times V$ is $\mu\star\beta$. Put $\varkappa$ equal to disintegration of $\tau$ with respect to $\mu$. Without loss of generality one can assume that $\varkappa\in\mathcal{D}_2[\beta]$. 
	
	Recall that
	$$m_i(\cdot)\triangleq m(\cdot,s_i,\mu_i,\varkappa_i). $$ 
	
	Set
	$$\operatorname{traj}^l\triangleq \operatorname{traj}^{s_{i_l}}_{m_{i_l}(\cdot)},\ \ \tau^l\triangleq \tau_{i_l},\ \  \gamma^l\triangleq \operatorname{traj}^l{}_\#\tau^l. $$

	Since the probabilities $\{\gamma^l\}$ are concentrated on the  set of $C_0$-Lipschtitz continuous functions from $[0,T]$ to $\td$ which is a compact subset of $\mathcal{C}$, $\{\gamma^l\}$ is relatively compact in $\mathcal{P}(\mathcal{C})$. Without loss of generality we can assume that $\{\gamma^l\}$ itself converges narrowly to some $\gamma\in\mathcal{P}(\mathcal{C})$. Moreover, $\gamma$ is concentrated on the set of $C_0$-Lipschitz continuous functions from $[0,T]$ to $\td$. 
	
	Put $\nu(t)\triangleq e_t{}_\#\gamma$. By (\ref{ineq:wass_motions}) we have that
	$$\lim_{l\rightarrow \infty}\sup_{t\in [0,T]}W_2(m_{i_l}(t),\nu(t))=0. $$
	
	It remains to prove that $$\nu(\cdot)=m(\cdot)\triangleq m(\cdot,s,\mu,\varkappa).$$ To this end let us show that $\nu(s)=\mu$ and $\gamma=\operatorname{traj}^*{}_\#\tau$, where 
	$\operatorname{traj}^*\triangleq \operatorname{traj}^{s}_{\nu(\cdot)}.$
	
	Since $m_{i_l}(s_{i_l})=\mu_i$ and $W_2(m_{i_l}(s),m_{i_l}(s_{i_l}))\leq C_0|s-s_{i_l}|$, we get that $\nu(s)=\mu$.
	
	Further, we have that, for any $y\in\td$, $\eta\in\mathcal{W}$,
	$$\|\operatorname{traj}^l(y,\eta)-\operatorname{traj}^*(y,\eta)\|\rightarrow 0\text{ as }l\rightarrow \infty. $$ The compactness of $\td$ and $\mathcal{W}$ implies that
	\begin{equation}\label{convergence:traj_l}
	\|\operatorname{traj}^l-\operatorname{traj}^*\|\triangleq\sup_{y\in\td}\sup_{\eta\in\mathcal{W}}\|\operatorname{traj}^l(y,\eta)-\operatorname{traj}^*(y,\eta)\|\rightarrow 0.
	\end{equation}
	
	Let $\phi\in C_b(\mathcal{C})$. Without loss of generality we can consider only the case when $\phi$ is $1$-Lipschitz continuous. We have that
	\begin{equation*}
	\begin{split}
	\Bigl|\int_{\mathcal{C}}&\phi(x(\cdot))\gamma^l(d(x(\cdot)))-\int_{\mathcal{C}}\phi(x(\cdot))(\operatorname{traj}^*{}_\#\tau)(d(x(\cdot)))\Bigr|\\ &=
	\Bigl|	\int_{\mathcal{C}}\phi(x(\cdot))(\operatorname{traj}^l{}_\#\tau^l)(d(x(\cdot)))- \int_{\mathcal{C}}\phi(x(\cdot))(\operatorname{traj}^*{}_\#\tau)(d(x(\cdot)))\Bigr| \\ &=
	\Bigl|	\int_{\td\times \mathcal{W}}\phi(\operatorname{traj}^l(y,\eta))\tau^l(d(y,\eta))-\int_{\td\times \mathcal{W}}\phi(\operatorname{traj}^*(y,\eta))\tau(d(y,\eta))\Bigr|\\ &=
	\Bigl|\int_{\td\times\mathcal{W}}\phi(\operatorname{traj}^l(y,\eta))\tau^l(d(y,\eta))-
	\int_{\td\times\mathcal{W}}\phi(\operatorname{traj}^*(y,\eta))\tau^l(d(y,\eta))\\ &{}\hspace{15pt}+
	\int_{\td\times\mathcal{W}}\phi(\operatorname{traj}^*(y,\eta))\tau^l(d(y,\eta))-
	\int_{\td\times\mathcal{W}}\phi(\operatorname{traj}^*(y,\eta))\tau(d(y,\eta))\Bigr|\\ &\leq
	\|\operatorname{traj}^l-\operatorname{traj}^*\|\\
	&{}\hspace{15pt}+
	\Bigl|\int_{\td\times\mathcal{W}}\phi(\operatorname{traj}^*(y,\eta))\tau^l(d(y,\eta))-
	\int_{\td\times\mathcal{W}}\phi(\operatorname{traj}^*(y,\eta))\tau(d(y,\eta))\Bigr|.
	\end{split}
	\end{equation*}
	
	Using (\ref{convergence:traj_l}) and the fact that $\tau^l$ converges narrowly to $\tau$ we get that $\gamma^l\rightarrow \operatorname{traj}^*{}_\#\tau$. At the same time $\gamma^l$ converges to $\gamma$. Thus, $\gamma=\operatorname{traj}^*{}_\#\tau$.  Taking into account the equality $\tau=\mu\star\varkappa$ we obtain the second statement of the lemma.
\end{proof}

The following lemmas are concerned with the sequences $\{\omega_k\}$ and $\{\omega^k\}$ Recall that $\{\omega_k\}$ is defined by (\ref{intro:omega_null}) and (\ref{intro:omega_k}), whereas the sequence $\{\omega^k\}$ is given by (\ref{intro:omega_v_null}), (\ref{intro:omega_v_k}).

Denote by $\operatorname{LCF}$ (respectively, $\operatorname{UCF}$) the set of all lower (respectively, upper) semicountinuous functions from $[0,T]\times \mathcal{P}^2(\td)$ to $\mathbb{R}$.

\begin{lemma}\label{lm:semicont}
	For any $k\in\mathbb{N}\cup\{0\}$, 
	$$\omega_k\in \operatorname{LCF}, \ \ \omega^k\in \operatorname{UCF}$$
\end{lemma}
\begin{proof}
	We  prove only the statement that $\omega_0\in\operatorname{LCF}$, as the statements for $\omega_k$, $k\in\mathbb{N}$ and $\omega^k$ can be obtained similarly. Let $s_i\rightarrow s$, $\mu_i\rightarrow \mu$. 
	Given $\varepsilon>0$, there exists $\beta_\varepsilon$ such that
	\begin{equation}\label{ineq:omega_0_choice}
	\omega_0(s,\mu)\leq \inf_{\varkappa\in\mathcal{D}_2[\beta_\varepsilon]}g(m(T,s,\mu,\varkappa))+\varepsilon.
	\end{equation} Let a sequence of distributions $\{\beta_{\varepsilon,i}\}\subset\mathcal{B}^c$ be such that $\mu_i\star\beta_{\varepsilon,i}$ converges narrowly to $\mu\star\beta_\varepsilon$. The existence of this sequence is ensured by   Lemma \ref{lm:cont_dependence}. Obviously,
	$$\omega_0(s_i,\mu_i)\geq \inf_{\varkappa\in\mathcal{D}_2[\beta_{\varepsilon,i}]}g(m(T,s_i,\mu_i,\varkappa)). $$  Choose $\varkappa_{\varepsilon,i}\in\mathcal{D}_2[\beta_{\varepsilon,i}]$ satisfying
	$$\omega_0(s_i,\mu_i)\geq g(m(T,s_i,\mu_i,\varkappa_{\varepsilon,i}))-\varepsilon. $$
	
	By Lemma \ref{lm:cont_dependence} and the continuity of $g$ there exists a distribution of controls $\varkappa_\varepsilon\in\mathcal{D}_2[\beta_\varepsilon]$ such that
	$$\liminf_{i\rightarrow\infty}\omega_0(s_i,\mu_i)\geq g(m(T,s,\mu,\varkappa_\varepsilon))-\varepsilon.  $$ Using (\ref{ineq:omega_0_choice}), we get
	$$\liminf_{i\rightarrow\infty}\omega_0(s_i,\mu_i)\geq\omega_0(s,\mu)-2\varepsilon.  $$ Passing to the limit when $\varepsilon\rightarrow 0$, we conclude that the function $\omega_0$ is lower semicontinuous.
\end{proof}

\begin{lemma}\label{lm:pim} There exists a limit 
	$$\omega_*(s,\mu)\triangleq \lim_{k\rightarrow\infty}\omega_k(s,\mu)<\infty. $$
	The function $\omega_*$ is $u$-stable. 
\end{lemma}
\begin{proof}
	First, notice that
	\begin{equation}\label{ineq:omega_k_k_plus}
	\omega_{k+1}(s,\mu)\geq \omega_k(s,\mu).
	\end{equation} Moreover,
	$$\omega_k(s,\mu)\leq \sup_{\varkappa\in \mathcal{D}}g(m(T,s,\mu,\varkappa))<\infty. $$ This and (\ref{ineq:omega_k_k_plus}) imply that
	the function $\omega_*$ is well-defined and proper.  Since the functions $\omega_k$ are lower semicontinuous (see Lemma \ref{lm:semicont}), $\omega_*$ is also  lower semicontinuous. 
	
	Further, for all $k\in\mathbb{N}$,
	$$\omega_*(T,\mu)=\omega_k(T,\mu)=g(T,\mu). $$
	
	It remains to prove that, for any $s\in [0,T]$, $\mu\in\mathcal{P}^2(\td)$, $\beta\in \mathcal{B}^c$, $r\in [s,T]$, there exists a distribution of controls $\varkappa\in\mathcal{D}_2[\beta]$ such that
	\begin{equation}\label{ineq:u_stability_star}
	\omega_*(s,\mu)\geq \omega_*(r,m(r,s,\mu,\varkappa)).
	\end{equation}
	Given $s,r\in [0,T]$, $s<r$, $\mu\in \mathcal{P}^2(\td)$ and $\beta\in\mathcal{B}^c$, let $\varkappa_k\in \mathcal{D}_2[\beta]$ be such that
	\begin{equation}\label{ineq:omega_k}
	\omega_{k}(s,\mu)\geq \omega_{k-1}(r,m(r,s,\mu,\varkappa_k)) -\frac{1}{k}.
	\end{equation} Using Lemma \ref{lm:cont_dependence}, we get that there exists a sequence $\{\varkappa_{k_l}\}$ and a distribution of controls $\varkappa_*\in\mathcal{D}_2[\beta]$ such that
	$$W_2(m(r,s,\mu,\varkappa_{k_l}),m(r,s,\mu,\varkappa_*))\rightarrow 0. $$ 
	
	Since the functions $\omega_k$ and $\omega_*$ are lower semicontinuous, the definition of $\omega_*$ means that
	\begin{equation}\label{equal:limit_omega}
	\mathrm{epi}(\omega_*)=\bigcap_{k=0}^\infty \mathrm{epi}(\omega_k)=\bigcap_{l=1}^\infty\mathrm{epi}(\omega_{k_l-1}).
	\end{equation} Hereinafter, $\operatorname{epi}(\varphi)$ stands for the epigraph of the functions $\varphi$. Moreover, the set $\operatorname{epi}(\omega_*)$ is closed. 
	
	If we denote $z^l\triangleq \omega_{k_l}(s,\mu)$, $z^*\triangleq \omega_*(s,\mu)$, $\nu^l\triangleq m(r,s,\mu,\varkappa_{k_l}),$ $\nu^*\triangleq m(r,s,\mu,\varkappa_*)$, then (\ref{ineq:omega_k}) implies that
	\begin{equation}\label{eneq:dist_epi}
	\mathrm{dist}((r,\nu^l,z^l),\mathrm{epi}(\omega_{k_l-1}))\leq \frac{1}{k_l}. 
	\end{equation} Here $\operatorname{dist}$ stands for the distance between a point and a set. Passing to the limit in (\ref{eneq:dist_epi}), and using  equality (\ref{equal:limit_omega}), we get that
	$$\mathrm{dist}((r,\nu^*,z^*),\mathrm{epi}(\omega_{*}))=0. $$ This is equivalent to (\ref{ineq:u_stability_star}). 
\end{proof}
Lemma \ref{lm:pim} and Theorem \ref{th:extremal} immediately imply
\begin{corollary}\label{cor:gamma_1_omega} For any $s\in [0,T]$, $\mu\in\mathcal{P}^2(\td)$,  $$\Gamma_1(s,\mu)\leq \omega_*(s,\mu).$$
\end{corollary}

The following statement is concerned with the sequence of functions $\{\omega^k\}_{k=0}^\infty$ introduced by (\ref{intro:omega_v_null}) and (\ref{intro:omega_v_k}).
\begin{lemma}\label{lm:omega_v_k_limit} For any $s\in [0,T]$, $\mu\in\mathcal{P}^2(\td)$, the limit $$\omega^*(s,\mu)\triangleq \lim_{k\rightarrow\infty}\omega^k(s,\mu)$$ is well-defined. The function $\omega^*$ is $v$-stable and
	$$\Gamma_2(s,\mu)\geq \omega^*(s,\mu),\ \ s\in [0,T], \ \ \mu\in\mathcal{P}^2(\td). $$
\end{lemma}
The proof of this Lemma is analogous to the proof of Lemma \ref{lm:pim} and Corollary~\ref{cor:gamma_1_omega}. To prove the lemma it suffices to replace $g$ with $-g$ and interchange the players.

The further construction requires the formalization of the mean field type differential game in the class of strategies with memory. 

Let $\mathcal{M}_s$ denote the set of continuous functions $m(\cdot):[s,T]\rightarrow \mathcal{P}^2(\td)$.

\begin{definition}\label{def:strategy_memory}
	Let $s\in [0,T]$. We say that a function $p:[s,T]\times \mathcal{M}_{s}\rightarrow \mathcal{A}^c$ is a stepwise strategy with memory of the first player on $[s,T]$ if there exists a partition $\Delta=\{t_i\}_{i=0}^N$ of $[s,T]$ such that, for any $m'(\cdot),m''(\cdot)\in\mathcal{M}_{s}$ and $t\in [s,T]$, the property that $t\in [t_{i},t_{i+1})$, $m'(t_0)=m''(t_0)$, \ldots, $m'(t_i)=m''(t_i)$ implies the equality
	$$p[t,m'(\cdot)]=p[t_i,m'(\cdot)]=p[t,m''(\cdot)]=p[t_i,m''(\cdot)]. $$
\end{definition}

\begin{definition}\label{def:motion_stepwise}
	Let $s\in [0,T]$, $\mu\in\mathcal{P}^2(\td)$, $p$ be a stepwise strategy with memory of the first player on $[s,T]$. Moreover, assume that the partition of $[s,T]$ $\Delta=\{t_i\}_{i=0}^N$ determines $p$. The flow of probabilities $m(\cdot)$ is called generated by $s$, $\mu$ and $p$ if  there exist distributions of controls $\varkappa_i\in\mathcal{D}$  such that $m(s)=\mu$ and, for $i=0,\ldots,n-1$, $t\in [t_i,t_{i+1})$,
	$$m(t)=m(t,t_i,m(t_i),\varkappa_i), \ \ \varkappa_i\in\mathcal{D}_1^0[p[t_i,m(\cdot)]]. $$
\end{definition}
Denote the set of  flows of probabilities generated by $s$, $\mu$ and $p$ by
$\widehat{\mathcal{X}}_1(s,\mu,p)$. Put
\begin{equation*}\label{intro:hat_gamma}
\widehat{\Gamma}_1(s,\mu)\triangleq\inf_{p}\sup_{m(\cdot)\in\widehat{\mathcal{X}}_1(s,\mu,p)}g(m(T)).
\end{equation*} Note that if $\mathfrak{u}$ is a feedback strategy, $\Delta$ is a partition of $[s,T]$, then the corresponding stepwise strategy of the first player is given by
$$p_{s,\mathfrak{u},\Delta}[t,m(\cdot)]\triangleq \mathfrak{u}[t_i,m(t_i)],\ \ t\in [t_i,t_{i+1}). $$
Thus, 
\begin{equation}\label{ineq:gamma_hat_gamma}
\Gamma_1(s,\mu)\geq \widehat{\Gamma}_1(s,\mu).
\end{equation}

The stepwise strategies with memory of the second player are introduced in the same way. A function $q:[s,T]\times \mathcal{M}_{s}\rightarrow \mathcal{B}^c$ is called a stepwise strategy with memory of the second player on $[s,T]$ if one can find a partition $\Delta=\{t_i\}_{i=0}^N$ of $[s,T]$ such that, for all $m'(\cdot),m''(\cdot)\in\mathcal{M}_{s}$ and $t\in [s,T]$, $$q[t,m'(\cdot)]=q[t_i,m'(\cdot)]=q[t,m''(\cdot)]=q[t_i,m''(\cdot)]. $$ when $t\in [t_{i},t_{i+1})$, $m'(t_0)=m''(t_0)$, \ldots, $m'(t_i)=m''(t_i)$.

The set of flows of probabilities generated by $s$, $\mu$ and $q$ is denoted by $\widehat{\mathcal{X}}_2(s,\mu,q)$. Put 
\begin{equation*}
\widehat{\Gamma}_2(s,\mu)\triangleq\sup_{q}\inf_{m(\cdot)\in\widehat{\mathcal{X}}_2(s,\mu,q)}g(m(T)).
\end{equation*}
Clearly, $\Gamma_2(s,\mu)\leq \widehat{\Gamma}_2(s,\mu)$, $\widehat{\Gamma}_1(s,\mu)\geq\widehat{\Gamma}_2(s,\mu)$. Combining this inequality with (\ref{ineq:gamma_hat_gamma}), we get
\begin{equation}\label{ineq:gamma_hat_2}
\Gamma_1(s,\mu)\geq \widehat{\Gamma}_1(s,\mu)\geq \widehat{\Gamma}_2(s,\mu)\geq \Gamma_2(s,\mu).
\end{equation}

\begin{lemma}\label{lm:gamma_2_omega}
	For any $s\in [0,T]$, $\mu\in\mathcal{P}^2(\td)$, 
	$$\Gamma_1(s,\mu)=\widehat{\Gamma}_1(s,\mu)=\widehat{\Gamma}_2(s,\mu)= \omega_*(s,\mu). $$
\end{lemma}
\begin{proof}
	
	Given $\varepsilon>0$, for any $s,\mu$, we define the sequence of stepwise strategies with memory of the second player $\{q^{k,\varepsilon}_{s,\mu}\}_{k=0}^\infty$ such that, for any $m(\cdot)\in \widehat{\mathcal{X}}_2(s,\mu,q^{k,\varepsilon}_{s,\mu})$,
	\begin{equation}\label{ineq:strategy_omega_k}
	g(m(T))\geq \omega_k(s,\mu)-\varepsilon\sum_{l=0}^k 2^{-l}.
	\end{equation}
	
	To define $q^{0,\varepsilon}_{s,\mu}$ choose $\beta_{s,\mu}^{0,\varepsilon}\in\mathcal{B}^c$  such that
	$$\omega_0(s,\mu)\leq \inf_{\varkappa\in \mathcal{D}_2[\beta_{s,\mu}^{0,\varepsilon}]}g(m(T,s,\mu,\varkappa))+\varepsilon. $$ Put 
	\begin{equation*}
	q^{0,\varepsilon}_{s,\mu}[t,m(\cdot)]\triangleq \beta_{s,\mu}^{0,\varepsilon}.
	\end{equation*}
	Obviously, inequality (\ref{ineq:strategy_omega_k}) is fulfilled for $k=0$.
	
	Now, assume that the strategies $q^{0,\varepsilon}_{s,\mu},\ldots, q^{n-1,\varepsilon}_{s,\mu}$ are already constructed for any $s\in [0,T]$, $\mu\in \mathcal{P}^2(\td)$ and satisfy (\ref{ineq:strategy_omega_k}) for $k=n-1$.

	Let $r^{n,\varepsilon}_{s,\mu}$ and $\beta^{n,\varepsilon}_{s,\mu}$ be such that
	$$\omega_{n}(s,\mu)\leq \inf_{\varkappa\in \mathcal{D}_2[\beta^{n,\varepsilon}_{s,\mu}]}\omega_{n-1}(r^{n,\varepsilon}_{s,\mu},m(r^{n,\varepsilon}_{s,\mu},s,\mu,\varkappa))+\varepsilon 2^{-n}. $$
	Put
	$$q^{n,\varepsilon}_{s,\mu}[t,m(\cdot)]\triangleq \left\{
	\begin{array}{ll}
	\beta^{n,\varepsilon}_{s,\mu}, & t\in [s,r_{s,\mu}^{n,\varepsilon})\\
	q^{n-1}_{r_{s,\mu}^{n,\varepsilon},m(r_{s,\mu}^{k,\varepsilon})}[t,m(\cdot)], & t\in [r_{s,\mu}^{n,\varepsilon},T]
	\end{array}\right. $$
	
	We have that 
	\begin{equation*}
	\omega_{n}(s,\mu)\leq \inf_{m(\cdot)\in \widehat{\mathcal{X}}(s,\mu,q_{s,\mu}^{n,\varepsilon})}\omega_{n-1}(r_{s,\mu}^{n,\varepsilon},m(r_{s,\mu}^{n,\varepsilon}))+\varepsilon2^{-n}.
	\end{equation*}
	Combining this and the fact that (\ref{ineq:strategy_omega_k}) is fulfilled for $k=n-1$, we obtain the truthfulness of (\ref{ineq:strategy_omega_k}) for $k=n$.
	
	Now, let $z<\omega_*(s,\mu)$. There exists $k$ such that $z<\omega_k(s,\mu)$. Let $\varepsilon$ be from $(0,(\omega_k(s,\mu)-z)/2)$. It follows from (\ref{ineq:strategy_omega_k}) that there exists a stepwise strategy of the second player $q=q_{s,\mu}^{k,\varepsilon}$ such that
	$$\inf_{m(\cdot)\in\widehat{\mathcal{X}}(s,\mu,q)}g(m(T))\geq \omega_k(s,\mu)-2\varepsilon>z. $$ This implies that, for any $z<\omega_*(s,\mu)$,
	$$\widehat{\Gamma}_2(s,\mu)>z. $$ Hence,
	$$\widehat{\Gamma}_2(s,\mu)\geq \omega_*(s,\mu). $$ This, Corollary \ref{cor:gamma_1_omega} together with inequality (\ref{ineq:gamma_hat_2}) give the conclusion of the lemma. 
\end{proof}
\begin{lemma}\label{lm:gamma_omega_v}
	For all $s\in [0,T]$, $\mu\in\mathcal{P}^2(\td)$, the following inequality holds:
	$$\Gamma_2(s,\mu)=\widehat{\Gamma}_2(s,\mu)=\widehat{\Gamma}_1(s,\mu)= \omega^*(s,\mu). $$	
\end{lemma}
The proof of this lemma is analogous to the proof of the previous lemma. It is based on the sequence $\{\omega^k\}$ introduced by (\ref{intro:omega_v_null}) and (\ref{intro:omega_v_k}).

\vspace{5pt}
\begin{proof}[Proof of Theorem   \ref{th:existence}]
The conclusion of Theorem \ref{th:existence} directly follows from Lemmas~\ref{lm:pim}, \ref{lm:omega_v_k_limit}, \ref{lm:gamma_2_omega}~and~\ref{lm:gamma_omega_v}.
\end{proof}

\section{Conclusion}
In the paper we construct the feedback strategies for the zero-sum first-order mean field type differential game. The construction relies on the notions of $u$- and $v$-stability. Furthermore, we prove the existence of the value function that is simultaneously $u$- and $v$-stable. These results can be regarded as a mean field analog of the Krasovskii-Subbotin theory for the finite dimensional differential games.

In the case of finite-dimensional differential game the $u$-stable (respectively, $v$-stable) function is a super- (respectively, sub-) solution to the corresponding Hamilton-Jacobi PDE. The extension of this result to the case of mean field type differential games is the subject of future work. Note that Cosso and Pham  studied  the second order mean field type differential game using nonanticipative strategy  (see \cite{Cosso_Pham}). In particular, they proved that the value function in this case is a viscosity solution of the corresponding Hamilton-Jacobi equation. This raises the question of the equivalence of the feedback  formalization and the approach based on  nonaticipative strategies.

In the paper we restrict our attention to the case of zero-sum games. Note that the case of nonzero-sum games is more complicated. In particular, the existence of Nash equilibria for finite dimensional differential games is proved only within the punishment approach \cite{Kononenko}, \cite{Tolwinski}.  Apparently, the punishment strategies can be also used to construct Nash equilibria for mean field type differential games.

\begin{acknowledgements} The author  would like to thank the anonymous referees for
	their valuable and helpful comments. 
\end{acknowledgements}

\bibliography{th_32_mfdg_dgaa}
\end{document}